\documentclass[12pt]{article}
\usepackage{amssymb}
\usepackage{amsmath}
\usepackage{amsfonts}
\usepackage{tikz}
\usepackage{wrapfig}
\usepackage{bookmark}

\newtheorem{theorem}{Theorem}

\newtheorem{proposition}[theorem]{Proposition}

\newenvironment{proof}[1][Proof]{\noindent\textbf{#1.} }{\ \rule{0.5em}{0.5em}}
\begin{document}

\title{Optimal tail comparison under convex majorization}
\author{Daniel J. Fresen\thanks{%
University of Pretoria, Department of Mathematics and Applied Mathematics,
daniel.fresen@up.ac.za}}
\date{}
\maketitle

\begin{abstract}
Following results of Kemperman and Pinelis, we show that if $X$ and $Y$ are real valued random variables such that $\mathbb{E}\left\vert Y\right\vert<\infty$ and for all non-decreasing convex $\varphi:\mathbb{R}\rightarrow [0,\infty)$, $\mathbb{E}\varphi(X)\leq\mathbb{E}\varphi(Y)$, then for all $s\in\mathbb{R}$ with $\mathbb{P}\left\{Y>s\right\}\neq 0$, $\mathbb{P}\left\{X\geq \mathbb{E}\left(Y:Y>s\right)\right\}\leq\mathbb{P}\left\{Y>s\right\}$. This bound is sharp in essentially the strictest possible sense: for any such $Y$ and $s$ there exists such an $X$ with $\mathbb{P}\left\{X\geq \mathbb{E}\left(Y:Y>s\right)\right\}=\mathbb{P}\left\{Y>s\right\}$.
\end{abstract}


\section{Introduction}

\noindent\underline{Majorization}

\smallskip

We consider the problem of estimating the cumulative distribution and
quantile function of a random variable $X$, given that $\mathbb{E}\varphi
\left( X\right) \leq \mathbb{E}\varphi \left( Y\right) $ for some random
variable $Y$ and all $\varphi\in\mathcal{F}$, where $\mathcal{F}$ is some collection of functions $\varphi:\mathbb{R}\rightarrow\mathbb{R}$ (we can think of $X$ as unknown and $Y$ as given). We do not insist that either $\mathbb{E}\varphi\left( X\right)$ or $\mathbb{E}\varphi\left( Y\right)$ are finite, although they should be well defined elements of $[-\infty,\infty]$ for all $\varphi\in\mathcal{F}$. This is a type of majorization; our interest stems from Pisier's version of the Gaussian concentration inequality, see \cite{Pis0}, where $\mathcal{F}$ is the collection of all convex functions. We refer the reader to \cite{MaOlAr} for a general theory of convex majorization.

\bigskip

\noindent\underline{The prototypical application: Gaussian concentration}

\smallskip

If $f:\mathbb{R}^n\rightarrow\mathbb{R}$ is a continuous function that is differentiable a.e. and $A$ and $B$ are independent random vectors in $\mathbb{R}^n$ each with the standard normal distribution, then there exists a real valued random variable $Z$ with the standard normal distribution in $\mathbb{R}$, independent of $A$ and $B$, such that for all convex $\varphi:\mathbb{R}\rightarrow\mathbb{R}$,
\begin{equation}
\mathbb{E}\varphi\left(f(A)-f(B)\right)\leq\mathbb{E}\varphi\left(\frac{\pi}{2}Z\left\vert \nabla f(A)\right\vert\right)\label{proto meijer}
\end{equation}
This is a simplified presentation of Pisier's result in \cite{Pis0}. One can bound the deviations of $f(A)$ about is median $\mathbb{M}f(A)$ in terms of the deviations of $f(A)-f(B)$ about $0$, and using (\ref{proto meijer}) one can then bound these deviations in terms of the deviations of $\frac{\pi}{2}Z\left\vert \nabla f(A)\right\vert$ about $0$. This is usually done by applying Markov's inequality, and the result is satisfactory for many purposes. However:

\noindent$\bullet$ The application of Markov's inequality would be ad hoc; the type of function $\varphi$ would depend on the distribution of $\frac{\pi}{2}Z\left\vert \nabla f(A)\right\vert$. While this is usually possible or even easy to implement, it would be better not to have to bother with it at all.

\noindent$\bullet$ The way Markov's inequality as usually applied, using exponential or power functions, one is typically not able to recover the correct order of magnitude of the tail probabilities. One can often do so up to a factor of $(1+o(1))$ in the exponent, but usually not up to a multiplicative factor of $C$.

\noindent$\bullet$ When the tails of $\left\vert \nabla f(A)\right\vert$ are very heavy, say polynomial, Markov's inequality combined with functions of the form $\varphi(x)=\left(\max\{0,x\}\right)^p$, for $p$ in a universally bounded interval, say $p\in(1,10)$, fails to recover non-trivial sub-Gaussian estimates in the central region of the distribution.

\bigskip

This paper plays a supporting role in a series of papers together with \cite{Fr20, FrIIa, FrIIb} that present novel applications of the Gaussian concentration inequality. There is not enough space here to go into a longer discussion, and we refer the reader to those papers and the references therein, as well as \cite[Section 5]{Pin98}, for more details.

\bigskip

\underline{Back to majorization}

\smallskip

For $x>0$ let
\[
\Lambda_r(x)=\left\{
\begin{array}{ccccc}
x^{-1/r}&:& r\in(0,\infty)\\
-\ln x&:& r=\infty
\end{array}%
\right.
\]
and set $\Lambda_r(0)=\infty$. It follows from Lemma 1 on p797 of \cite{ScWe} (due to Kemperman) that if $r\in(1,\infty]$, and $X$ and $Y$ are non-negative random variables such that $\mathbb{E}\varphi\left( X\right) \leq \mathbb{E}\varphi \left( Y\right) $ for all functions of the form $\varphi(x)=\max\left\{x-a,0\right\}$ ($a>0$), and we assume that $\Lambda_r\left(\mathbb{P}\left\{Y\geq t\right\}\right)$ is a convex function of $t\geq 0$, then for all such $t$,
\[
\mathbb{P}\left\{X\geq t\right\}\leq C(r)\mathbb{P}\left\{Y\geq t\right\}
\]
where
\[
C(r)=\left\{
\begin{array}{ccc}
\left(1-\frac{1}{r}\right)^{-r} &:& r\in(1,\infty)\\
e &:& r=\infty
\end{array}
\right.
\]
This is generalized by Pinelis \cite{Pin98, Pin99} who considers the case where $0<\alpha<r\leq\infty$ and $\mathcal{F}=\mathcal{F}_\alpha$ ($\alpha>0$) is the collection of non-decreasing $\varphi:\mathbb{R}\cup\{-\infty\}\rightarrow\mathbb{R}$ such that $\left(\varphi-\varphi(-\infty)\right)^{1/\alpha}$ is convex on $\mathbb{R}$. See in particular Theorem 3.11 in \cite{Pin98} and Theorem 4 in \cite{Pin99} for his main results in this direction. As noted in  \cite[Remark 3.13]{Pin98}, the convexity requirement on the tails of $Y$ can be relaxed at the cost of optimality. The coefficient $C(r)$ in Kemperman's result, and a similar coefficient $C(r,\alpha,\beta)$ in results of Pinelis, are sharp in the sense that they are pointwise least possible among all such functions of 1 (resp. 3) variables.

\bigskip

\underline{More on this paper}

\smallskip
In what follows we present an alternative approach to the theory of tail comparison inequalities under convex majorization under minimal assumptions and yielding results that are optimal not only among a class of functions, but for each choice of $Y$ and for each value of $t=\mathbb{E}\left(Y:Y>s\right)$ (for suitable $s\in\mathbb{R}$, see Proposition \ref{conditional bound}).

When applied to the Gaussian concentration inequality one has two options:

\noindent$\bullet$ More precise but less explicit estimates comparing the tail probabilities of $f(A)-\mathbb{M}(A)$ to those of $\frac{\pi}{2}Z\left\vert \nabla f(A)\right\vert$, see Proposition \ref{conditional bound} and more explicit estimates in Propositions \ref{prec prob} and \ref{areeae}, or

\noindent$\bullet$ Instant order of magnitude bounds for the quantiles of $f(A)-\mathbb{M}(A)$ under mild assumptions on the tails of $\frac{\pi}{2}Z\left\vert \nabla f(A)\right\vert$; see Proposition \ref{Gaussian v}).

\bigskip

In either case one avoids the ad hoc application of Markov's inequality; it is still used in the background, with a function of the form $\varphi(x)=\max\{0,x-b\}$, but is confined to the proof of Proposition \ref{conditional bound}.

\section{Notation and basic comments}

$\mathbb{M}$ denotes median, $\mathbb{E}$ expectation, and $C$, $c$, etc. universal constants. For a real valued random variable $Y$ with distribution $\mu $ on the Borel subsets of $\mathbb{R}$, the cumulative
distribution $F_Y:\mathbb{R}\rightarrow \left[ 0,1\right] $ is defined by $%
F_Y(t)=\mu \left( -\infty ,t\right] $. The generalized inverse $F_Y^{-1}:\left(
0,1\right) \rightarrow \mathbb{R}$ defined by $F_Y^{-1}(s)=\inf \left\{ t\in \mathbb{R}:F_Y(t)\geq s\right\}$ is known as the quantile function and also denoted $H_Y$. $H_Y$ is non-decreasing, left continuous, and if $U$ is a random variable uniformly distributed in $(0,1)$, then $%
H_Y\left( U\right) $ has distribution $\mu $. It will be convenient to phrase certain results in terms of the function $H_Y^*(x)=H_Y(1-x)$, $x\in(0,1)$. Since $(0,1)$ endowed with Lebesgue measure on its Borel subsets is a probability space, $H_Y$ is in fact a random variable in its own right, also with distribution $\mu$, and is therefore its own quantile function. $\mu$ is non-atomic if and only if $H_Y$ is strictly increasing. For each $t\in\mathbb{R}$, let $E_t$ be the collection of all non-negative convex functions $\varphi :\mathbb{R}\rightarrow \mathbb{R}$ that are strictly increasing on $\left[ t,\infty \right) $, and such that $\varphi
(t)\neq 0$. Let $E^*_t$ be the collection of all functions $\varphi _{t,a}:\mathbb{R}\rightarrow\mathbb{R}$ of the form $\varphi _{t,a}(x)=\max \left\{ 0,a\left( x-t\right) +1\right\} $, for $a\in \left( 0,\infty \right) $. Let
\[
E=\cup_{t\in\mathbb{R}}E_t\hspace{2cm}E^*=\cup_{t\in\mathbb{R}}E_t^*
\]
so $E$ is the collection of all non-negative convex functions from $\mathbb{R}$ to $\mathbb{R}$, excluding the constant functions.

\section{Results}

\begin{proposition}
\label{minimizer}Let $\mu $ be a probability measure on $\mathbb{R%
}$ with $\int_{[0,\infty)}x d\mu (x)<\infty $, let 
$Y$ be a random variable with distribution $\mu $, and consider any $t\in \mathbb{R}$.

\smallskip

\noindent I. For all $\varphi\in E_t$ there exists $\varphi _{t,a}\in E^*_t$ such that
\[
\frac{\mathbb{E}\varphi_{t,a} \left( Y\right) }{\varphi_{t,a} (t)}\leq\frac{\mathbb{E}\varphi \left( Y\right) }{\varphi (t)}
\]
\noindent II. If we assume, in addition, that $\mu$ is non-atomic, that $t>\mathbb{E}Y$, and that $\mu\left((t,\infty)\right)>0$, then the function%
\begin{equation*}
\varphi \mapsto \frac{\mathbb{E}\varphi \left( Y\right) }{\varphi (t)}
\end{equation*}%
defined on $E_t$ and taking values in $\left[ 0,\infty \right] $ achieves a global
minimum at some $\varphi _{t,a}\in E^*_t$, where $a\in \left( 0,\infty \right) $ is such that%
\begin{equation}
\int_{t-a^{-1}}^{\infty }\left( x-t\right) d\mu (x)=0\label{min cinditii}
\end{equation}%
Any such $a$ defines a minimizer, and at least one such $a$ exists. The corresponding minimum is
\begin{equation}
\frac{\mathbb{E}\varphi_{t,a} \left( Y\right) }{\varphi_{t,a} (t)}=\mathbb{P}\{Y>t-a^{-1}\}\label{minimuumu}
\end{equation}
\end{proposition}

\begin{proof}
Replacing $\varphi $ with $\varphi /\varphi (t)$, we may assume that $%
\varphi (t)=1$, and so we add this condition to the constraints on $\varphi $%
. By convexity $\varphi (x)\geq \max \left\{ 0,\varphi ^{\prime }(t)\left(
x-t\right) +1\right\} $, where $\varphi ^{\prime }(t)$ here denotes $%
\lim_{h\rightarrow 0^{+}}\left( \varphi \left( t+h\right) -\varphi \left(
t\right) \right) /h$ (necessarily exists and is finite), and so $\mathbb{E}%
\varphi \left( Y\right) \geq \mathbb{E}\max \left\{ 0,\varphi ^{\prime
}(t)\left( Y-t\right) +1\right\} $. This proves \textit{I} and implies that we may restrict our attention to $E^*_t$. Differentiating under the integral sign, for $a\in \left( 0,\infty \right) $ 
\begin{equation*}
\frac{d}{da}\int_{\mathbb{R}}\max \left\{ 0,a\left( x-t\right) +1\right\}
d\mu (x)=\int_{t-a^{-1}}^{\infty }\left( x-t\right) d\mu (x)
\end{equation*}%
which is a continuous non-decreasing function of $a$. By the assumption $t>%
\mathbb{E}Y$, this function is negative for some $a\in \left( 0,\infty
\right) $ and converges to $\int_{t}^{\infty }\left( x-t\right) d\mu (x)>
0$ as $a\rightarrow \infty $. Therefore there exists $a>0$ such that $%
\int_{t-a^{-1}}^{\infty }\left( x-t\right) d\mu (x)=0$, and the convex
function $a\mapsto \int_{\mathbb{R}}\max \left\{ 0,a\left( x-t\right)
+1\right\} d\mu (x)$ achieves a global minimum over $\left( 0,\infty \right) 
$ at this value of $a$. (\ref{minimuumu}) follows from (\ref{min cinditii}).
\end{proof}

\begin{proposition}\label{conditional bound}
Let $X$ and $Y$ be real valued random variables such that $\mathbb{E}\max\{0,Y\}<\infty$ and for all $\varphi_a\in E^*$, $\mathbb{E}\varphi_a(X)\leq\mathbb{E}\varphi_a(Y)$. Consider any $s\in\mathbb{R}$.

\noindent I. If $\mathbb{P}\{Y>s\}=0$ then $\mathbb{P}\{X>s\}=0$.

\noindent II. If $\mathbb{P}\{Y>s\}\neq0$ then $\mathbb{P}\left\{X\geq \mathbb{E}\left(Y:Y>s\right)\right\}\leq\mathbb{P}\{Y>s\}$.
\end{proposition}

\begin{proof}
Consider the collection (which is non-empty) of all $t\in\mathbb{R}$ such that $t>s$ and $\mathbb{E}\left[Y1_{(s,\infty)}(Y)\right]\leq t\mathbb{P}\{Y>s\}$. Momentarily consider any such $t$. Setting $\varphi(x)=\max\{(x-s)/(t-s),0\}$,
\begin{eqnarray*}
&&\mathbb{P}\{X\geq t\}=\mathbb{P}\{\varphi(X)\geq \varphi(t)\}\leq\frac{\mathbb{E}\varphi(X)}{\varphi(t)}\leq\mathbb{E}\varphi(Y)=\int_{(s,\infty)}\frac{x-s}{t-s}d\mu(x)\\
&&=\int_{(s,\infty)}\frac{x-t}{t-s}d\mu(x)+\mathbb{P}\{Y>s\}=\frac{\mathbb{E}\left[Y1_{(s,\infty)}(Y)\right]- t\mathbb{P}\{Y>s\}}{t-s} +\mathbb{P}\{Y>s\}
\end{eqnarray*}
This last quantity is at most $\mathbb{P}\{Y>s\}$. If $\mathbb{P}\{Y>s\}=0$ then $\mathbb{P}\{X\geq t\}=0$ for all $t>s$, so by continuity of measures, $\mathbb{P}\{X>s\}=0$. This proves I. If $\mathbb{P}\{Y>s\}\neq0$, we choose $t=\mathbb{E}\left[Y1_{(s,\infty)}(Y)\right]/\mathbb{P}\{Y>s\}$ which proves II.
\end{proof}

\begin{proposition}\label{sharpness p}
Let $Y$ be a real valued random variable such that $\mathbb{E}\left\vert Y\right\vert<\infty$ and consider any $s\in\mathbb{R}$ such that $\mathbb{P}\{Y\geq s\}\neq0$. Then there exists a real valued random variable $X$ such that $\mathbb{E}\varphi(X)\leq \mathbb{E}\varphi(Y)$ for all convex $\varphi$ and $\mathbb{P}\left\{X\geq \mathbb{E}\left(Y:Y\geq s\right)\right\}=\mathbb{P}\{Y\geq s\}$.
\end{proposition}

\begin{proof}
Let $X=Y1_{(-\infty,s)}(Y)+\mathbb{E}\left(Y:Y\geq s\right)1_{[s,\infty)}(Y)$. For any convex $\varphi$, by Jensen's inequality $\mathbb{E}\varphi(X)=\mathbb{E}\left[\varphi(Y)1_{(-\infty,s)}(Y)\right]+\mathbb{P}\{Y\geq s\}\varphi\left(\mathbb{E}\left(Y:Y\geq s\right)\right)$ which is
\begin{eqnarray*}
&\leq&\mathbb{E}\left[\varphi(Y)1_{(-\infty,s)}(Y)\right]+\mathbb{P}\{Y\geq s\}\mathbb{E}\left(\varphi(Y):Y\geq s\right)\\
&=&\mathbb{E}\left[\varphi(Y)1_{(-\infty,s)}(Y)\right]+\mathbb{E}\left[\varphi(Y)1_{[s,\infty)}(Y)\right]=\mathbb{E}\varphi(Y)
\end{eqnarray*}
\end{proof}

\begin{proposition}\label{maximal quantile}
Let $X$ and $Y$ be real valued random variables such that $\mathbb{E}\max\{0,Y\}<\infty$ and for all $\varphi\in E^*$, $\mathbb{E}\varphi(X)\leq\mathbb{E}\varphi(Y)$. For all $x\in(0,1)$ the quantile function of $X$ obeys
\begin{equation}
H_X(x)\leq\frac{1}{1-x} \int_x^1H_Y(u)du\label{quant max bound}
\end{equation}
\end{proposition}

\begin{proof}
We start by assuming that $H_Y$ is strictly increasing. Using Proposition \ref{conditional bound} with $s=H_Y(x)$,
\begin{equation}
\mathbb{P}\left\{X\geq \mathbb{E}\left(Y:Y>H_Y(x)\right)\right\}\leq\mathbb{P}\left\{Y>H_Y(x)\right\}\label{praesti}
\end{equation}
Since $H_Y$ as a random variable defined on $(0,1)$ has the same distribution as $Y$, we may evaluate probabilities and expected values involving $Y$ by using $H_Y$ instead. Now $\left\{H_Y>H_Y(x)\right\}=(x,1)$, so $RHS$ of the above equation is $1-x$. On the other hand,
\[
\mathbb{E}\left(Y:Y>H_Y(x)\right)=\mathbb{E}\left(H_Y:H_Y>H_Y(x)\right)=\frac{1}{1-x}\int_x^1H_Y(u)du
\]
So (\ref{praesti}) can be re-written as
\[
\mathbb{P}\left\{X<\frac{1}{1-x}\int_x^1H_Y(u)du\right\}\geq x
\]
which implies (\ref{quant max bound}) since an infimum is always a lower bound. If $H_Y$ is not strictly increasing, define $H^\sharp(x)=H_Y(x)+\varepsilon x$, for $x\in(0,1)$, and note that $H^\sharp$ is strictly increasing and left continuous, so it is therefore the quantile function of some random variable (it is in fact its own quantile function). Now apply what has been proved to $X$ and $H^\sharp$ and take $\varepsilon\rightarrow 0^+$.
\end{proof}

\begin{proposition}\label{mag of quanti}
If in Proposition \ref{maximal quantile} we have $H_Y(x)\leq H(x)$ for all $x\in(0,1)$ and some function $H:(0,1)\rightarrow(0,\infty)$, and for all $x,\delta\in(0,1)$, $H\left(1-\delta(1-x)\right)\leq \omega(\delta)H(x)$ where $\omega:(0,1)\rightarrow(1,\infty)$ is any non-increasing function such that $\int_0^1\omega(u)du<\infty$, then for all $x\in(0,1)$,
\[
H_X(x)\leq H(x)\int_0^1\omega(u)du
\]
\end{proposition}

\begin{proof}
This follows by writing
\begin{eqnarray*}
H_X(x)&\leq&\frac{1}{1-x} \int_x^1H(u)du=\frac{1}{1-x} \int_x^1H\left(1-\frac{1-u}{1-x}(1-x)\right)du\\
&\leq&\frac{1}{1-x} \int_x^1H(x)\omega\left(\frac{1-u}{1-x}\right)du=H(x)\int_0^1\omega(u)du
\end{eqnarray*}
\end{proof}

\begin{proposition}\label{Gaussian v}
Consider the setting of Proposition \ref{maximal quantile}, and suppose that $T,\gamma\geq 1$, $p>1$, and $Q:(0,\infty)\rightarrow(0,\infty)$ is a function that satisfies
\begin{equation}
Q(t)\exp\left(\frac{-t^2}{2p}\right)\leq TQ(s)\exp\left(\frac{-s^2}{2p}\right)\label{Q inc rate}
\end{equation}
for all $0<s<t$. If $\mathbb{P}\left\{Y>Q(t)\right\}< \gamma\exp\left(-t^2/2\right)$ for all $t>0$, then for all $t>0$
\begin{equation}
\mathbb{P}\left\{X>\frac{pT}{p-1}Q(t)\right\}\leq \gamma\exp\left(-t^2/2\right)\label{gaussian v bound X}
\end{equation}
\end{proposition}

\begin{proof}
Define $H(x)=Q\left(\sqrt{2\ln (\gamma/(1-x))}\right)$ and $\omega(x)=Tx^{-1/p}$ for $x\in(0,1)$, so $\mathbb{P}\left\{Y>H(x)\right\}<1-x$ and by definition of $H_Y$, $H_Y(x)\leq H(x)$. For all $\delta\in(0,1)$, set
\[
t=\sqrt{2\ln\frac{\gamma}{1-x}} \hspace{2cm} t+s=\sqrt{2\ln\frac{\gamma}{\delta(1-x)}}
\]
It now follows from (\ref{Q inc rate}) that $H(1-\delta(1-x))\leq T\delta^{-1/p}H(x)$, so by Proposition \ref{mag of quanti}, for all $t>0$ such that $\gamma\exp\left(-t^2/2\right)<1$,
\[
H_X\left(1-\gamma\exp\left(-t^2/2\right)\right)\leq \frac{pT}{p-1}H\left(1-\gamma\exp\left(-t^2/2\right)\right)=\frac{pTQ(t)}{p-1}
\]
(\ref{gaussian v bound X}) now follows since $H_X$ is non-decreasing and has the same distribution as $X$.
\end{proof}

One can show that when $Y(x)=x^{-1}\left(\log\left(e/x\right)\right)^{-2}$, $x\in(0,1)$,
\[
E_n=\left\{x\in(0,1):e^{-n^2}\leq x< e^{-(n-1)^2}\right\}
\]
$\mathcal{T}=\sigma\left(\left\{E_n:n\in\mathbb{N}\right\}\right)$ and $X=\mathbb{E}\left(Y:\mathcal{T}\right)$, then $\mathbb{E}\left\vert Y\right\vert<\infty$ and for all convex $\varphi:\mathbb{R}\rightarrow\mathbb{R}$, $\mathbb{E}\varphi(X)\leq\mathbb{E}\varphi(Y)$, yet
\begin{equation}
\limsup_{x\rightarrow 1^{-}}\frac{H_X(x)}{H_Y(x)}=\infty\label{lim supper}
\end{equation}
We now shift focus from quantiles to tail probabilities. Recall $H_Y^*(x)=H_Y(1-x)$.

\begin{proposition}\label{prec prob}
Let $X$ and $Y$ be real valued random variables such that $\mathbb{E}\left\vert Y \right\vert<\infty$ and for all non-decreasing convex $\varphi:\mathbb{R}\rightarrow\mathbb{R}$, $\mathbb{E}\varphi(X)\leq\mathbb{E}\varphi(Y)$. Let $x\in(0,1)$, $R>1$ and assume that $H_Y^*\left(x/R\right)\geq(1/x)\int_0^xH_Y^*(u)du$. Then,
\[
\mathbb{P}\left\{X> t\right\}\leq R\mathbb{P}\left\{Y\geq t\right\}
\]
where $t=(1/x)\int_0^xH_Y^*(u)du$. Note: by continuity, for all $t>\mathbb{E}Y$ such that $\mathbb{P}\{Y>t\}\neq0$, there exists $x\in(0,1)$ such that $t=(1/x)\int_0^xH_Y^*(u)du$.

\end{proposition}

\begin{proof}
By Proposition \ref{maximal quantile}, $H_X^*(x)\leq t$, so using the fact that $X$ and $H_X^*$ have the same distribution, $\mathbb{P}\left\{X>t\right\}\leq\mathbb{P}\left\{X>H_X^*(x)\right\}=\mathbb{P}\left\{H_X^*>H_X^*(x)\right\}\leq x$. Using the fact that $Y$ and $H_Y^*$ have the same distribution,
\[
R\frac{x}{R}\leq R\mathbb{P}\left\{H_Y^*\geq H_Y^*(x/R)\right\}= R\mathbb{P}\left\{Y\geq H_Y^*(x/R)\right\}\leq R\mathbb{P}\left\{Y\geq t\right\}
\]
\end{proof}

We shall consider the following condition: For all $\delta,x,y\in(0,1)$,
\begin{equation}
\left\vert H_Y^*\left(\delta x\right)-H_Y^*\left(\delta y\right)\right\vert\leq T\delta^{-1/p}\left\vert H_Y^*\left(x\right)-H_Y^*\left(y\right)\right\vert\label{deviation of quant fn}
\end{equation}
where $p>1$ and $T\geq 1$. This is implied by conditions of the form
\begin{eqnarray}
\frac{d}{dx}\frac{1-F_Y(x)}{F_Y'(x)}\leq \frac{1}{p}&:&x\in F_Y^{-1}\left((0,1)\right)\label{Iee}\\
\frac{H_Y''(t)}{H_Y'(t)}\leq \left(1+\frac{1}{p}\right)(1-t)^{-1}&:&t\in(0,1)\label{IIee}\\
H_Y'(1-\delta(1-x))\leq T\delta^{-1-1/p}H_Y'(x)&:&t\in(0,1)\label{IIIee}
\end{eqnarray}
Precise regularity conditions aside, (\ref{Iee})$\Leftrightarrow$(\ref{IIee}) by the inverse function theorem, (\ref{IIee})$\Rightarrow$(\ref{IIIee}) by recognizing a logarithmic derivative, and (\ref{IIIee})$\Rightarrow$(\ref{deviation of quant fn}) by FTC.

\begin{proposition}\label{areeae}
Let $Y$ be a real valued random variable such that $H_Y$ is convex and such that (\ref{deviation of quant fn}) holds for some $(p,T)\in(1,\infty)\times[1,\infty)$ and all $\delta,x,y\in(0,1)$. Consider any $R>1$ such that $T\leq(R-1)R^{-1/p}/2$. Then for all $x\in(0,1)$, $H_Y^*\left(x/R\right)\geq(1/x)\int_0^xH_Y^*(u)du$ (as required in Proposition \ref{prec prob}).
\end{proposition}

\begin{proof}

\centering
\setlength{\unitlength}{1cm} \thicklines
\begin{tikzpicture}
\draw [-stealth] (0,0) -- (5,0);
\draw [-stealth] (0,0) -- (0,2.6);
\draw (1,2.8) arc (215:265:5.4);
\draw [dashed] (4.5,0) -- (4.5,1.72);
\draw [dashed] (0,0.578) -- (4.5,0.578);
\draw [dashed] (2,0) -- (2,1.72);
\draw [dashed] (0,1.72) -- (4.5,1.72);
\node at (4.5,-0.3){$x$};
\node at (2.0,-0.3){$x/R$};
\node at (-0.75,0.578){$H_Y^*(x)$};
\node at (-1,1.72){$H_Y^*(x/R)$};
\node at (3.8,1.3){$A$};
\node at (2.35,1.0){$B$};
\node at (1.06,1.14){$C$};
\node at (0.6,2.2){$D$};
\node at (1.85,2.25){$H_Y^*$};

\node at (2,-1){Figure 1: The graph of $H_Y^*$ and various regions of interest.};
\node at (2,-1.55){The letters $A, B, C, D$ denote the regions as well as their areas.};

\end{tikzpicture}

\raggedright


\bigskip

If $H_Y^*\left(x/R\right)=(1/x)\int_0^xH_Y^*(u)du$ then $D=A$. We show that $D\leq A$, which implies
\[
H_Y^*\left(x/R\right)\geq(1/x)\int_0^xH_Y^*(u)du
\]
By the change of variables $s=Ru$ and (\ref{deviation of quant fn}),
\[
\int_0^{x/R}\left(H_Y^*(u)-H_Y^*(x/R)\right)du\leq TR^{-1+1/p}\int_0^{x}\left(H_Y^*(s)-H_Y^*(x)\right)ds
\]
i.e. $D\leq TR^{-1+1/p}(B+C+D)$. By convexity and direct computation $B\leq A$ and $C=\left(A+B+C\right)/R$, so $C=\left(A+B\right)/(R-1)$ and $C\leq 2A/(R-1)$. By assumption, $TR^{-1+1/p}<1$, so
\[
D\leq \frac{TR^{-1+1/p}}{1-TR^{-1+1/p}}(B+C)\leq \frac{R+1}{(R-1)(T^{-1}R^{1-1/p}-1)}A
\]
By assumption, the coefficient of $A$ here is at most 1, i.e. $D\leq A$.
\end{proof}

Proposition \ref{areeae} has been simplified for brevity; one only needs conditions on the far right tail of the distribution, the assumption of convexity is hardly utilized, and the crude estimate $B\leq A$ can often be improved to give better dependence of $R$ on $T$ and $p$. For example, when $r>1$ and $t\mapsto\left(\mathbb{P}\left\{Y\geq t\right\}\right)^{-1/r}$ is convex, which is the basic assumption used by Kemperman and Pinelis, this function lies above its tangent lines, so  $\mathbb{P}\left\{Y\geq t\right\}$ lies below a corresponding tangent power-function with exponent $-r$, and reflecting about the line $y=x$, $H_Y^*$ lies below a tangent power-function with exponent $-1/r$.

\bigskip

\centering
\setlength{\unitlength}{1cm} \thicklines
\begin{tikzpicture}

\draw [-stealth] (-4.5,0) -- (-0.75,0);
\draw [-stealth] (-4.5,0) -- (-4.5,2.6);
\draw (-3.9,2.3) arc (212:238:7);
\draw [dashed] (-1.77,0) -- (-1.77,1.07);
\draw [dashed] (-2.925,0) -- (-2.925,1.07);
\draw [dashed] (-4.5,1.07) -- (-1.77,1.07);
\node at (-2.15,0.7){$A$};
\node at (-4,1.5){$D$};
\node at (-3,2.1){$H_Y^*$};
\node at (-1.75,-0.3){$x$};
\node at (-2.9,-0.3){$x/R$};

\draw [-stealth] (0,0) -- (3.75,0);
\draw [-stealth] (0,0) -- (0,2.6);
\draw (0.6,2.3) arc (212:238:7);
\draw (0.65,2.53) arc (200:250:4);
\draw [dashed] (2.73,0) -- (2.73,1.07);
\draw [dashed] (1.575,0) -- (1.575,1.07);
\draw [dashed] (0,1.07) -- (2.73,1.07);
\node at (2.75,-0.3){$x$};
\node at (1.6,-0.3){$x/R$};

\draw [-stealth] (4.5,0) -- (8.25,0);
\draw [-stealth] (4.5,0) -- (4.5,2.6);
\draw (5.15,2.53) arc (200:250:4);
\draw [dashed] (7.23,0) -- (7.23,1.07);
\draw [dashed] (6.075,0) -- (6.075,1.07);
\draw [dashed] (4.5,1.07) -- (7.23,1.07);
\node at (6.9,0.75){$A'$};
\node at (5,1.5){$D'$};
\node at (7.25,-0.3){$x$};
\node at (6.1,-0.3){$x/R$};

\node at (2,-1){Figure 2: $H_Y^*$ on the left, which lies below its tangent};

\node at (2,-1.5){power-function with exponent $-1/r$ on the right.};

\end{tikzpicture}

\raggedright

It then follows, setting $R=(1-1/r)^{-r}$, that $D\leq D'=A'\leq A$, so $D\leq A$ as required in the proof of Proposition \ref{areeae}, and we recover Kemperman's result (details left to the reader).

\end{document}